\documentclass{amsart}

\usepackage{amssymb,mathrsfs,amscd,graphicx, array}
\allowdisplaybreaks
\usepackage[nocompress, noadjust]{cite}  %Sorting citations. See
\usepackage{enumitem}
\usepackage{mathtools}
\makeatletter
\newcommand*{\Relbarfill@}{\arrowfill@\Relbar\Relbar\Relbar}
\newcommand*{\xeq}[2][]{\ext@arrow 0055\Relbarfill@{#1}{#2}}
\makeatother
\usepackage[
  pdfencoding=auto, % or unicode
  psdextra,
]{hyperref}
\usepackage{letltxmacro}
\LetLtxMacro{\oldsqrt}{\sqrt}
\renewcommand{\sqrt}[2][]{\,\oldsqrt[#1]{#2}\,}
\makeatletter
\def\@tocline#1#2#3#4#5#6#7{\relax
  \ifnum #1>\c@tocdepth % then omit
  \else
    \par \addpenalty\@secpenalty\addvspace{#2}%
    \begingroup \hyphenpenalty\@M
    \@ifempty{#4}{%
      \@tempdima\csname r@tocindent\number#1\endcsname\relax
    }{%
      \@tempdima#4\relax
    }%
    \parindent\z@ \leftskip#3\relax \advance\leftskip\@tempdima\relax
    \rightskip\@pnumwidth plus4em \parfillskip-\@pnumwidth
    #5\leavevmode\hskip-\@tempdima
      \ifcase #1
       \or\or \hskip 1em \or \hskip 2em \else \hskip 3em \fi%
      #6\nobreak\relax
    \dotfill\hbox to\@pnumwidth{\@tocpagenum{#7}}\par
    \nobreak
    \endgroup
  \fi}
\makeatother
\usepackage{bbm}
\newcommand{\bsh}{\backslash}

\makeatletter

\def\greekbolds#1{%
 \@for\next:=#1\do{%
    \def\X##1;{%
     \expandafter\def\csname V##1\endcsname{\boldsymbol{\csname##1\endcsname}}
     }
   \expandafter\X\next;
  }
}

\greekbolds{alpha,beta,iota,gamma,lambda,nu,eta,Gamma,varsigma}

\def\make@bb#1{\expandafter\def
  \csname bb#1\endcsname{{\mathbb{#1}}}\ignorespaces}

\def\make@bbm#1{\expandafter\def
  \csname bb#1\endcsname{{\mathbbm{#1}}}\ignorespaces}

\def\make@bf#1{\expandafter\def\csname bf#1\endcsname{{\bf
      #1}}\ignorespaces} 

\def\make@gr#1{\expandafter\def
  \csname gr#1\endcsname{{\mathfrak{#1}}}\ignorespaces}

\def\make@scr#1{\expandafter\def
  \csname scr#1\endcsname{{\mathscr{#1}}}\ignorespaces}

\def\make@cal#1{\expandafter\def\csname cal#1\endcsname{{\mathcal
      #1}}\ignorespaces} 
\def\do@Letters#1{#1A #1B #1C #1D #1E #1F #1G #1H #1I #1J #1K #1L #1M
                 #1N #1O #1P #1Q #1R #1S #1T #1U #1V #1W #1X #1Y #1Z}
\def\do@letters#1{#1a #1b #1c #1d #1e #1f #1g #1h #1i #1j #1k #1l #1m
                 #1n #1o #1p #1q #1r #1s #1t #1u #1v #1w #1x #1y #1z}
\do@Letters\make@bb   \do@letters\make@bbm
\do@Letters\make@cal  
\do@Letters\make@scr 
\do@Letters\make@bf \do@letters\make@bf   
\do@Letters\make@gr   \do@letters\make@gr
\makeatother

\newcommand{\sel}{\mathrm{sel}}
\newcommand{\non}{\mathrm{non}}

\newcommand{\abs}[1]{\lvert #1 \rvert}
\newcommand{\zmod}[1]{\mathbb{Z}/ #1 \mathbb{Z}}

\newcommand{\wh}{\widehat}

\newcommand{\scc}{\mathrm{sc}}
\DeclareMathOperator{\Mass}{Mass}

\DeclareMathSymbol{\twoheadrightarrow} {\mathrel}{AMSa}{"10}

\DeclareMathOperator{\Hom}{Hom}

\DeclareMathOperator{\Gal}{Gal}

\DeclareMathOperator{\Nr}{Nr}
\DeclareMathOperator{\Cl}{Cl}

\DeclareMathOperator{\Emb}{Emb}
\DeclareMathOperator{\Frac}{Frac}

\newcommand{\Z}{\mathbb Z}
\newcommand{\Q}{\mathbb Q}

\newcommand{\wcO}{\widehat{\mathcal{O}}}
\newcommand{\whD}{\widehat{D}}
\newcommand{\whF}{\widehat{F}}

\newcommand{\whO}{\widehat{O}}

\newcounter{thmcounter} 
\numberwithin{thmcounter}{section}  % place this command in different
\newtheorem{thm}[thmcounter]{Theorem}
\newtheorem{lem}[thmcounter]{Lemma}

\newtheorem{prop}[thmcounter]{Proposition}
\theoremstyle{definition}

\numberwithin{equation}{section}
\numberwithin{figure}{section}
\numberwithin{table}{section}

\newtheoremstyle{notitle}  % this product a paragraph that's
  {}% measure of space to leave above the theorem. E.g.: 3pt
  {}% measure of space to leave below the theorem. E.g.: 3pt
  {\itshape}% name of font to use in the body of the theorem
  {}% measure of space to indent
  {}% name of head font
  {\ }% punctuation between head and body
  {.5em}% space after theorem head; " " = normal interword space
  {}% Man
\theoremstyle{notitle}

 \title[Divisibility]{On the divisibility of the class numbers of
    quaternion orders}
\author{Yucui Lin}

\address{(Lin) School of
  Mathematics and Statistics, Wuhan University, Luojiashan, 430072,
  Wuhan, Hubei, P.R. China}   

\email{yucui.lin@whu.edu.cn}

 \author{Jiangwei Xue}

\address{(Xue) Collaborative Innovation Center of Mathematics, School of
  Mathematics and Statistics, Wuhan University, Luojiashan, 430072,
  Wuhan, Hubei, P.R. China}   

\address{(Xue) Hubei Key Laboratory of Computational Science (Wuhan
  University), Wuhan, Hubei,  430072, P.R. China.} 

\email{xue\_j@whu.edu.cn}

\begin{document}
\date{\today} 
\begin{abstract}
Let $F$ be a number field, and $D$ be a quaternion $F$-algebra.  We
show that the class number of any residually unramified $O_F$-order
(e.g.~an Eichler order) in $D$ is divisible by the class number of
$F$. 
\end{abstract}

\maketitle
%\tableofcontents   %% this add the table of contents after the abstract.
%%%%%%%%%%            How to add to content line              %%%%%%%%%%%
%\addcontentsline{toc}{chapter}{\protect\numberline{}Appendix}

%%%%%%%%%%%%%%%%%              START HERE        %%%%%%%%%%%%%%%%%%%%%%%%
%%%%%%%%%%%%%%%%%%%%%%%%%%%%%%%%%%%%%%%%%%%%%%%%%%%%%%%%%%%%%%%%%%%%%%%%%

%\linenumbers

\section{Introduction}

%The algebra
%$D$ admits a canonical involution $\alpha\mapsto \bar \alpha$ such
%that $\Tr(\alpha)=\alpha+\bar \alpha$ and
%$\Nr(\alpha)=\alpha \bar \alpha$ are respectively the \emph{reduced trace}
%and \emph{reduced norm} of $\alpha\in D$.

Let $F$ be a number field, and $D$ be a quaternion $F$-algebra.
Let
$O_F$ be the ring of integers of $F$, and $\calO$ be an $O_F$-order
(of full rank) in
$D$.  For each finite prime $\grp$ of $F$, we write
$\calO_\grp$ for
the $\grp$-adic completion of $\calO$, and $\grJ(\calO_\grp)$ for the 
Jacobson radical of $\calO_\grp$. Following \cite[Definition~24.3.2]{voight-quat-book},
we say $\calO$ is \emph{residually unramified at $\grp$} if
$\calO_\grp/\grJ(\calO_\grp)$ is not equal to the finite residue field
$O_F/\grp$. If $\calO$ is residually unramified at every
finite prime $\grp$ of $F$, then we simply say that $\calO$ is \emph{residually
unramified}. For example, every Eichler order is residually
unramified by \cite[Lemma~24.3.6]{voight-quat-book}.

 % two locally principal (fractional) right $\calO$-ideals $I$ and $I'$
% in $D$ belong to \emph{the same ideal class} if they are isomorphic as
% right $\calO$-modules (equivalently, if there exists $\alpha\in
% D^\times$ such that $I'=\alpha I$. Let 

By definition, the class number $h(\calO)$ is the cardinality of
the finite set $\Cl(\calO)$ of locally principal nonzero right $\calO$-ideal
classes in $D$. 
If we write $\wcO$ for the profinite completion
of $\calO$, and  $\whD$ for the ring of finite adeles of $D$, then
$\Cl(\calO)$ may be described adelically as 
\begin{equation}
  \label{eq:189}
  \Cl(\calO)\simeq D^\times\bsh
    \whD^\times/\wcO^\times.   
  \end{equation}
  It is well known that $h(\calO)$ depends only on the genus  of
$\calO$. In other words, if $\calO'$ is another $O_F$-order in $D$
\emph{belonging to the same
genus} as $\calO$ (that is, $\calO'_\grp$ is isomorphic to
$\calO_\grp$ for every finite prime $\grp$ of $F$), then  $h(\calO')=h(\calO)$.  As usual, the
class number of $F$ is denoted by $h(F)$.

% Given an $O_F$-order $B$ (of full rank) in 
% a finite field extension $K/F$,  we write $h(B)$ for the class
% number of $B$, that is, $h(B):=\abs{\Cl(B)}$, where $\Cl(B)$ is the
% ideal class group (i.e.~Picard
% group) of $B$ as in \cite[Definition~12.5]{Neukirch-ANT}. For simplicity,  put
% $h(K):=h(O_K)$ as usual. 

The main result of this paper is as follows. 
\begin{thm}\label{thm:main}
If $\calO$ is residually unramified, then $h(\calO)$ is divisible by
$h(F)$. 
\end{thm}
 There are multiple layers to Theorem~\ref{thm:main}. To explain
them, assume that $\calO$ is residually unramified henceforth.  The reduced norm map $\Nr: D^\times\to F^\times$ induces the following surjective map: 
\begin{equation}\label{eq:4}
\Psi: D^\times\bsh
    \whD^\times/\wcO^\times\xrightarrow{\Nr} F_D^\times\bsh \whF^\times/\Nr(\wcO^\times).
\end{equation}
Here $F_D^\times:=\Nr(D^\times)$, which coincides with the subgroup of $F^\times$ consisting of the
elements that are positive at each infinite place of $F$ ramified in
$D$ by the Hasse-Schilling-Maass theorem
\cite[Theorem~33.15]{reiner:mo}
\cite[Theorem~III.4.1]{vigneras}. Moreover, since $\calO$ is
residually unramified, it has been shown in the proof of
\cite[Lemma~2.17]{Xue-Yu-Selec-2022} that 
\begin{equation}
  \label{eq:1}
  \Nr(\wcO^\times)=\whO_F^\times. 
\end{equation}

If $D$ satisfies the Eichler condition (that is, there exists an infinite
place of $F$ that is split in $D$), then $\Psi$ is in fact
bijective by the strong approximation
theorem \cite[III.4.3]{vigneras}.  See \cite[Corollaire~III.5.7]{vigneras} for the proof in the Eichler order case, which applies
to the current setting as well.  Following \cite[\S III.5,
p.~88]{vigneras}, we define the \emph{restricted class number of $F$
  with respect to $D$} as
$h_D(F):=\abs{\whF^\times/(F_D^\times\whO_F^\times)}$.  
Thus if $D$ satisfies the Eichler condition, then $h(\calO)=h_D(F)$,
which is obviously divisible by $h(F)$.

Now suppose  that $D$ does not satisfy the Eichler
condition. Necessarily, $F$ is a totally real field, and $D$ is
ramified at all infinite places of $F$. Such quaternion algebras are
called \emph{totally definite}.  In this case, $F_D^\times$ is just the group  $F_+^\times$ of totally positive
elements of $F^\times$, so the right hand side of  (\ref{eq:4})
coincides with the  
narrow class group $\Cl^+(O_F)$, and the map 
$\Psi$  can be  interpreted concretely as follows
\begin{equation}
  \label{eq:5}
  \Psi: \Cl(\calO)\to \Cl^+(O_F), \qquad [I]\mapsto [\Nr(I)]^+.   
\end{equation}
Here for each nonzero locally principal right $\calO$-ideal
$I\subset D$, we
write $[I]$ for its right $\calO$-ideal class, and $[\Nr(I)]^+$ for the narrow
$O_F$-ideal class of $\Nr(I)$. The map $\Psi$ remains surjective but is
no longer  injective
in general.

The properties of totally definite
quaternion $F$-algebras closely mirror those of CM-extensions of $F$
in many aspects.  For example, if $[F:\Q]$ is odd, then every
CM-extension $K/F$ is ramified\footnote{See
  \cite[Proposition~3.1]{Gross2003Rama} for the proof by B.H.~Gross
  and \cite[Corollary~1.3]{YangTonghai-2005PAMQ} by Tonghai Yang. A
  simple proof by David Rohrlich is also reproduced in
  \cite[p.~327]{YangTonghai-2005PAMQ}.}  at some finite prime of $F$.
Similarly, the same parity condition on $[F:\Q]$
implies that  every totally definite quaternion $F$-algebra $D$ is
ramified at some finite prime of $F$ since the total number of places
of $F$ ramified in $D$ has to be even.  In another aspect, if $F$ is
an arbitrary totally real field and $K/F$ is a CM-extension, then
$h(K)$ is divisible by $h(F)$ by
\cite[Theorem~4.10]{Washington-cyclotomic}.  Moreover, if $K/F$ is
further assumed to be ramified at some finite prime of $F$, then the same
proof (replacing the Hilbert class field with the narrow Hilbert class
field) shows that $h(K)$ is divisible by the \emph{narrow class
  number} $h^+(F)$.  
An analogues divisibility result to this latter case holds for totally definite
quaternion algebras by
\cite[Theorem~1.3]{xue-yu:spinor-class-no},  namely, if $D$ is ramified at some finite prime
of $F$, then $h(\calO)$ is divisible by $h^+(F)$. Indeed, it was shown in
\cite{xue-yu:spinor-class-no} that this ramification assumption
on $D$ forces the
fibers of $\Psi$ to share the same cardinality, and hence $h(\calO)$ is
divisible by $h^+(F)$.

Unfortunately, this divisibility result does not hold as soon as the
 ramification assumption on $D$ is dropped. For example, let
$F=\Q(\sqrt{7})$, and $D$ be the totally definite quaternion
$F$-algebra that is unramified at all the finite primes of $F$. Then 
$h(\calO)=3$ for any maximal order $\calO$ in $D$ by \cite[Table~1,
p.~676]{xue-yang-yu:ECNF}, which is not divisible by
$h^+(F)=2$. Thus to prove Theorem~\ref{thm:main}, we still need to treat the
following 
remaining case. 

\begin{prop}\label{prop:III}
  Let $F$ be a totally real number field of even degree over $\Q$, and $D$ be the unique totally definite quaternion
$F$-algebra that is unramified at all the finite primes of
$F$. Suppose that $\calO$ is residually unramified. Then $h(\calO)$ is
divisible by $h(F)$. 
\end{prop}

Proposition~\ref{prop:III} was first claimed by Vign\'eras
in \cite[Remarque, p.~82]{vigneras:ens} for maximal orders.  She
studied the following  natural multiplicative action of $\Cl(O_F)$ on $\Cl(\calO)$:
\begin{equation}
  \label{eq:3}
  \Cl(O_F)\times \Cl(\calO)\to \Cl(\calO), \qquad ([\gra],
  [I])\mapsto [\gra I].   
\end{equation}
 Vign\'eras 
observed that under the assumption of Proposition~\ref{prop:III}, the
number of orbits of this action coincides with the type number
 $t(\calO)$ when $\calO$ is maximal. However, this action is not
necessarily free as Vign\'eras's argument had suggested. Indeed, 
it was later pointed out by Ponomarev \cite[concluding remarks,
p.~103]{ponomarev:aa1981} that  $h(\calO)/h(F)$ is not necessarily
equal to the type number of $\calO$, so 
Vign\'eras's proof needs some adjustments. A complete proof for
integrality of $h(\calO)/h(F)$ based on a
refinement of Vigneras's ideas was supplied by Chia-Fu Yu and 
the second named author  in \cite[\S5]{xue-yu:type_no}. 
However, this proof makes explicit use of the maximality of
$\calO$ and cannot be easily generalized to non-maximal orders.

In this paper, we take a different approach  and make use of the class
number formulas developed in \cite{xue-yu:spinor-class-no}.  Consider
the following composition of $\Psi:\Cl(\calO)\to \Cl^+(O_F)$ with the canonical projection map $\pi:\Cl^+(O_F)\to
\Cl(O_F)$: 
\begin{equation}\label{eq:7}
    \Phi=\pi\circ\Psi: \Cl(\calO)\to \Cl(O_F), \qquad [I]\mapsto [\Nr(I)]. 
\end{equation}
We shall show that the fibers of $\Phi$ share the
same cardinality,  thus proving Proposition~\ref{prop:III}.  As in the
proof of \cite[Theorem~4.10]{Washington-cyclotomic}, the divisibility
of $h(K)$ by $h(F)$ for a CM-extension $K/F$ stems from the linear
disjointness of $K$ from the Hilbert class field $H$ of $F$. It turns
out that the same principle lies at the root of the divisibility of $h(\calO)$
by $h(F)$ as well. The details of the proof will be worked out
in Section~\ref{sec:proof}.

\section{The proof of Proposition~\ref{prop:III}}
\label{sec:proof}

Keep the notation and assumption of Proposition~\ref{prop:III}
throughout this section.  By definition, each fiber of the map $\Phi$
in \eqref{eq:7} is a disjoint union of fibers of $\Psi: \Cl(\calO)\to \Cl^+(O_F)$ over a coset
of $\ker(\pi)$.  A formula for the cardinality of each fiber of $\Psi$
has been produced in \cite[Theorem~1.3]{xue-yu:spinor-class-no}. Thus
it is enough to show that the summation of such formulas over each
coset of $\ker(\pi)$ is constant,  that is,  independent of the choice
of the coset.

Let $h_\scc(\calO)$ be the cardinality of
the neutral fiber of $\Psi$. From
\cite[(2.12)]{xue-yu:spinor-class-no}, for any 
nonzero locally principal right $\calO$-ideal $I$, we have 
\begin{equation}
  \label{eq:6}
  \abs{\Psi^{-1}([\Nr(I)]^+)}=h_\scc(\calO_l(I)),  
\end{equation}
where $\calO_l(I):=\{x\in D:  xI\subseteq I\}$ is the left order of
$I$.  The order $\calO_l(I)$ belongs to the same genus as
$\calO$. Indeed, since $I$ is locally principal, its profinite
completion $\wh{I}$ can be written as $x\wcO$ for some $x\in
\whD^\times$, which implies that 
\begin{equation}
\wh{\calO_l(I)}=x \wcO x^{-1}.   
\end{equation}
Let
$\whD^1:=\ker(\whD^\times\xrightarrow{\Nr} \whF^\times)$ be the
reduced norm one subgroup of $\whD^\times$. Two members $[I], [I']\in
\Cl(\calO)$ lie in the same fiber of $\Psi$ if and only if there
exists $y\in D^\times \whD^1$ such that $\wh{I}'=y\wh{I}$, in which
case
\begin{equation}
  \label{eq:9}
\wh{\calO_l(I')}= y\,\wh{\calO_l(I)}\, y^{-1}.     
\end{equation}

To write down a formula for $h_\scc(\calO_l(I))$ explicitly, we set up some
notations related to certain quadratic
$O_F$-orders called \emph{CM $O_F$-orders}. 
Since $D$ is totally definite, a quadratic field extension $K/F$
embeds into $D$ only if $K/F$ is a CM-extension.  An
$O_F$-order $B$ of full rank in a CM-extension
of $F$ 
 will be  called a \emph{CM $O_F$-order}.  Let $h(B)$ be the class
 number of $B$, and $w(B)$ be the unit group index $[B^\times:
 O_F^\times]$. According to  \cite[Remarks, p.~92]{Pizer1973}
 (cf.~\cite[\S3.1]{li-xue-yu:unit-gp} and
 \cite[\S3.3]{xue-yang-yu:ECNF}), 
there are only finitely many CM $O_F$-orders $B$ satisfying $w(B)>1$,  
so we collect them into a finite set $\scrB$.

For each CM $O_F$-order $B$ with fractional field $K$, we write
$\Emb(B, \calO)$ for the set of \emph{optimal embeddings} of $B$ into
$\calO$, that is,
\begin{equation}
  \label{eq:22}
  \Emb(B, \calO)=\{\varphi\in \Hom_F(K, D): \varphi(K)\cap \calO=\varphi(B)\}.
\end{equation}
Similarly, at each finite prime $\grp$ of $F$, we write
$\Emb(B_\grp, \calO_\grp)$ for the set of optimal embeddings of
$B_\grp$ into $\calO_\grp$.  The unit group $\calO_\grp^\times$ acts
on $\Emb(B_\grp, \calO_\grp)$ from the right by conjugation, and for
almost all $\grp$ there
is a unique orbit  by
\cite[Theorem~II.3.2]{vigneras}. Following \cite[\S V.2,
p.~143]{vigneras}, we put
\begin{equation}
  M(B):=\frac{h(B)}{w(B)}\prod_\grp \abs{\Emb(B_\grp,
    \calO_\grp)/\calO_\grp^\times},  
\end{equation}
where the product runs over all finite prime $\grp$ of $F$.  From the
definition,  $M(B)$ depends only on $B$ and the
genus of $\calO$.

Clearly, if
$\Emb(B, \calO)\neq \emptyset$, then
$\Emb(B_\grp, \calO_\grp)\neq \emptyset$ for every  $\grp$. Conversely, if
$\Emb(B_\grp, \calO_\grp)\neq \emptyset$ for every $\grp$, then there exists
an order $\calO_0$ in the same genus as $\calO$ such that
$\Emb(B, \calO_0)\neq \emptyset $ by
\cite[Corollary~30.4.18]{voight-quat-book}. Given $[I]\in \Cl(\calO)$,
we define a symbol as follows:
% \footnote{In light of (\ref{eq:9}) and
% \cite[Definition~1.2]{xue-yu:spinor-class-no}, the definition here
% matches with the one  in \cite[(1.9)]{xue-yu:spinor-class-no}.}
\begin{equation}
  \label{eq:10}
  \Delta(B, \calO_l(I))=
  \begin{dcases}
    1 & \parbox[t]{180pt}{if  there exists $ [I']\in \Cl(\calO)$ such that
      $\Psi([I'])=\Psi([I])$ and $\Emb(B, \calO_l(I'))\neq
      \emptyset$,}\\[4pt]
    0 & \text{otherwise.}
  \end{dcases}
\end{equation}
In light of (\ref{eq:9}) and
\cite[Definition~1.2]{xue-yu:spinor-class-no}, the definition here
matches with the one in \cite[(1.9)]{xue-yu:spinor-class-no}. Particularly,  we recover the definition of
$\Delta(B,\calO)$ in \cite[(1.9)]{xue-yu:spinor-class-no} by taking
$I=\calO$ in \eqref{eq:10}.  By
definition, $\Delta(B, \calO_l(I))$ takes the same value on the fibers
of $\Psi$, so it descends along $\Psi: \Cl(\calO)\to \Cl^+(O_F)$ to a
map
\begin{equation}\label{eq:20}
\delta:  \Cl^+(O_F)\to \{0, 1\}, \qquad [\Nr(I)]^+\mapsto \Delta(B,
  \calO_l(I)). 
\end{equation}
Since every order in the same genus as $\calO$ is isomorphic to some
$\calO_l(I)$, the map $\delta$ takes constant
value $0$ if and only if there exists $\grp$ such that $\Emb(B_\grp,
\calO_\grp)=\emptyset$ (equivalently, $M(B)=0$).
If the map $\delta$ is surjective, then we say that $B$ is
\emph{optimally spinor selective} (\emph{selective} for short) for the
genus of $\calO$.  We shall see from the formula of
$h_\scc(\calO_l(I))$ in (\ref{eq:8}) that such orders $B$ are really the
culprits behind the variation in size of the fibers of $\Psi$.
Since
$D$ is unramified at all finite primes of $F$ and $\calO$ is
residually unramified by assumption, we get the following
characterization of  CM $O_F$-orders that are selective for the genus
of $\calO$
by combining \cite[Theorem~2.5 and
Lemma~2.7]{xue-yu:spinor-class-no}.  

% To characterize the , we write $\grd(\calO)$ for the
% \emph{reduced discriminant} \cite[\S I.4, p.~24]{vigneras} of
% $\calO$. For each finite prime $\grp$ of $F$, let 

\begin{lem}\label{lem:sel}
  Let $B$ be a CM $O_F$-order with fractional field $K$. Suppose 
  that $\Emb(B_\grp, \calO_\grp)\neq \emptyset$ for every finite prime
  $\grp$ of $F$ (that is, $M(B)\neq 0$). Then $B$ is  selective for the genus of
  $\calO$ if and only if both of the following conditions hold:
  \begin{enumerate}[label=(\alph*)]
  \item $K/F$ is unramified at every finite prime of $F$,
  \item   if $\grp$ is a finite prime of $F$ with $\nu_\grp(\grd(\calO))\equiv
  1\pmod{2}$, then $\grp$ splits in $K$. 
\end{enumerate}
Here  $\nu_\grp: F^\times\twoheadrightarrow \Z$ denotes the normalized discrete
valuation of $F$ attached to $\grp$,  and $\grd(\calO)$ denotes the 
\emph{reduced discriminant} of $\calO$.

Moreover, if $B$ is selective for the genus of
$\calO$, then
\begin{equation}\label{eq:13}
  \Delta(B, \calO_l(I))=(\Nr(I), K/F)+\Delta(B, \calO), 
\end{equation}
where $(\Nr(I), K/F)$ is the Artin symbol\footnote{Here the Artin
  symbol is well-defined since $K/F$ is unramified at all finite
  primes of $F$ by part (a) of the lemma.} of $\Nr(I)$ in $\Gal(K/F)$,
and the summation  is taken within $\zmod{2}$ with the
canonical identification $\Gal(K/F)\cong \zmod{2}$. 
\end{lem}

Given a CM $O_F$-order $B$ with fractional field $K$, we put
\begin{equation}
s(B, \calO):=  \begin{cases}
    1 & \text{if $K$ satisfies  conditions (a) and (b) in
      Lemma~\ref{lem:sel},}\\
    0 & \text{otherwise.} 
  \end{cases}
\end{equation}
By definition, $s(B, \calO)$ depends only on  $K$ and
the genus of $\calO$.

With the notations $M(B), \Delta(B, \calO_l(I))$ and $s(B,\calO)$ in
place, we are now ready to recall the formula for
$h_\scc(\calO_l(I)$ from   
\cite[Theorem~1.3]{xue-yu:spinor-class-no}: 
\begin{equation}
  \label{eq:8}
  \begin{split}
        h_\scc(\calO_l(I))=&\Mass_\scc(\calO_l(I))\\ &+\frac{1}{2h^+(F)}\sum_{B\in
          \scrB}2^{s(B, \calO)}\Delta(B, \calO_l(I))(w(B)-1)M(B).
\end{split}        
    \end{equation}
Here $\Mass_\scc(\calO_l(I))$ denotes the \emph{spinor class
        mass} of $\calO$, which is defined in
      \cite[(3.3)]{xue-yu:spinor-class-no} and can be computed by
      formula \cite[(3.8)]{xue-yu:spinor-class-no}. From
      \cite[Lemma~3.1]{xue-yu:spinor-class-no},
      $\Mass_\scc(\calO_l(I))$ depends only on the genus of
      $\calO_l(I)$. In particular, 
      $\Mass_\scc(\calO_l(I))=\Mass_\scc(\calO)$, which is independent
      of the choice of $[I]\in \Cl(\calO)$.

To see the effects of the selective CM $O_F$-orders on the class
number $h_\scc(\calO_l(I))$ more clearly, we
partition the set $\scrB$ into two subsets 
\begin{equation}
  \scrB_\sel:=\{B\in \scrB:  s(B,\calO)=1 \text{ and } M(B)\neq
  0\}, \qquad \scrB_\non:=\scrB\smallsetminus \scrB_\sel.  
\end{equation}
If $B\in \scrB_\non$, then $\Delta(B, \calO_l(I))$ takes constant
value $0$ or $1$ for all  $[I]\in \Cl(\calO)$ depending on whether
$M(B)$ is zero or not. It follows that $\Delta(B,
\calO_l(I))M(B)=M(B)$ for $B\in \scrB_\non$. 
Thus the formula for $h_\scc(\calO_l(I))$ may be reorganized as
follows:
\begin{equation}\label{eq:11}
    \begin{split}
        h_\scc(\calO_l(I))=&\Mass_\scc(\calO)+\frac{1}{2h^+(F)}\sum_{B\in
          \scrB_\non}(w(B)-1)M(B)\\ &+\frac{1}{h^+(F)}\sum_{B\in
          \scrB_\sel}\Delta(B, \calO_l(I))(w(B)-1)M(B).
\end{split}        
\end{equation}
Here the first line on the right hand side of (\ref{eq:11}) depends
only on the genus of $\calO$ and is independent of the choice of
$[I]\in \Cl(\calO)$. It is clear from \eqref{eq:20} and (\ref{eq:11}) that as $[\Nr(I)]^+$
varies in $\Cl^+(O_F)$,   the variation in values of $\Delta(B,
\calO_l(I))$ for  $B\in
\scrB_\sel$ is  what causes the fibers of $\Psi:\Cl(\calO)\to
\Cl^+(O_F)$ to possibly 
vary in sizes.

%Whether and when a given CM $O_F$-order $B$ is optimally spinor selective
% for the genus of $\calO$ is the central question of \emph{optimal
%   spinor selectivity theory}
% \cite{M.Arenas-et.al-opt-embed-trees-JNT2018, Xue-Yu-Selec-2022, 
%   Maclachlan-selectivity-JNT2008, peng-xue:select}

% Clearly,  if 
% then $\Delta(B, \calO_l(I))=0$ for every $[I]\in \Cl(\calO)$. 
%  it turns out that for some CM $O_F$-order $B$, it may happen that
% $\Delta(B, \calO_l(I))=1$ for some but not all $[I]\in
% \Cl(\calO)$. Such orders

% As shown in \cite{xue-yu:spinor-class-no}, the reason that the  is caused by a subset of
% CM $O_F$-orders in $\scrB$ that are optimally spinor selective for the
% genus of $\calO$. 

%   Consider the subset of $\scrB$ consisting of all CM
% $O_F$-orders $B$ satisfying 

% Given a 
% nonzero locally principal right $\calO$-ideal $I$,
% we write $\calO_l(I)$ for the 
% If we denote the cardinality of
% the central fiber of $\Psi$ by  $h_\scc(\calO)$, then from
% \cite[(2.12)]{xue-yu:spinor-class-no}, 

Now consider the map $\Phi=\pi\circ\Psi: \Cl(\calO)\to \Cl(O_F)$ as
defined in (\ref{eq:7}), where $\pi: \Cl^+(O_F)\to \Cl(O_F)$ is the
canonical projection.  Given an $O_F$-ideal class
$[\gra]\in \Cl(O_F)$, our goal is to show that the size of the fiber
$\Phi^{-1}([\gra])$ is independent of the choice of $[\gra]$.
Actually, this is automatic if $\pi$ is an isomorphism. Indeed, for
every $B\in \scrB_\sel$,  its fractional field $K$ is contained in  the narrow Hilbert class
field $H^+$ of $F$ by part (a) of 
Lemma~\ref{lem:sel}. If $\pi$ is an isomorphism, then $H^+$ coincides
with the (wide) Hilbert class field $H$ of $F$, which is totally
real and cannot contain any CM subfield. Thus in this case
$\scrB_\sel=\emptyset$, which implies that
$h_\scc(\calO_l(I))$ is a constant independent of the choice of
$[I]\in \Cl(\calO)$.

For the remaining discussion, assume that $\pi$ is non-isomorphic
and put $r:=\abs{\ker(\pi)}$. 
By definition,
$\Phi^{-1}([\gra])$ is a disjoint union of $r$ distinct fibers of $\Psi$ over $\pi^{-1}([\gra])$. 
Let $\{[I_1], \cdots, [I_r]\}\subseteq
\Cl(\calO)$ be a complete set of representatives for such fibers, with
one ideal class for each fiber.  From (\ref{eq:6}), we have
\begin{equation}\label{eq:12}
  \abs{\Phi^{-1}([\gra])}=\sum_{i=1}^r h_\scc(\calO_l(I_i)).
\end{equation}
Combining (\ref{eq:11}) and (\ref{eq:12}), to show that 
  $\abs{\Phi^{-1}([\gra])}$ is independent of the choice of $[\gra]\in
  \Cl(O_F)$, it is enough to prove the following lemma.

  \begin{lem}
    For each $B\in \scrB_\sel$, we have $\sum_{i=1}^r \Delta(B,
    \calO_l(I_i))=r/2$. 
  \end{lem}
  \begin{proof}
Put $K=\Frac(B)$ and identify $\Gal(K/F)$ with 
$\zmod{2}$. From (\ref{eq:13}), we get
    \begin{equation}\label{eq:16}
      \sum_{i=1}^r\Delta(B, \calO_l(I_i))=\abs{\{1\leq i\leq r:
        (\Nr(I_i), K/F)=1-\Delta(B, \calO)\}}.   
    \end{equation}
As mentioned previously,  $K$ is contained in the narrow Hilbert class
field $H^+$ of $F$. 
If we put  $\sigma_i=(\Nr(I_i), H^+/F)\in \Gal(H^+/F)$ for each
    $1\leq i\leq r$, then $(\Nr(I_i),
K/F)=\sigma_i|_K$, the restriction of $\sigma_i$ to $K$. 

%     For simplicity, put $\varepsilon:=1-\Delta(B, \calO)\in \{0,
%     1\}$, and 
% Since $K\subseteq H^+$, we have  

On the other hand,    by the choices of $[I_i]\in \Cl(\calO)$, we
have
\begin{equation}
  \label{eq:14}
\pi^{-1}([\gra])=\{[\Nr(I_1)]^+, \cdots, [\Nr(I_r)]^+\}\subseteq \Cl^+(O_F).   
\end{equation}
In light of the following commutative diagram of Artin isomorphisms
\[
  \begin{CD}
    \Cl^+(O_F)@>{\simeq}>> \Gal(H^+/F)\\
    @V{\pi}VV  @VV{\text{restriction}}V\\
    \Cl(O_F) @>{\simeq}>> \Gal(H/F),
  \end{CD}
\]
the set $\{\sigma_1, \cdots, \sigma_r\}$ coincides with 
$  \{\sigma \in \Gal(H^+/F):
  \sigma|_H=(\gra, H/F)\}$. 
Combining this with
(\ref{eq:16}), we obtain
\begin{equation*}\label{eq:18}
    \sum_{i=1}^r\Delta(B, \calO_l(I_i))=\abs{\{\sigma\in
      \Gal(H^+/F): \sigma|_H=(\gra, H/F) \text{ and } 
        \sigma|_K=1-\Delta(B, \calO)\}}.    
\end{equation*}
Both of the restriction maps $\Gal(H^+/F)\twoheadrightarrow \Gal(H/F)$ and
$\Gal(H^+/F)\twoheadrightarrow \Gal(K/F)$ factor through
$\Gal(H^+/F)\twoheadrightarrow \Gal(HK/F)$, where $HK$ denotes the compositum of $H$
and $K$ inside $H^+$.  Since the (totally real) Hilbert class field $H/F$ is linearly
disjoint over $F$ with the CM-extension $K/F$, there is a canonical 
isomorphism   
\[\Gal(HK/F)\simeq \Gal(H/F)\times \Gal(K/F). \]
It follows that
\[
  \begin{split}
\sum_{i=1}^r\Delta(B, \calO_l(I_i))&=
\abs{\ker(\Gal(H^+/F)\to \Gal(HK/F))}\\&=\frac{1}{2}\abs{\ker(\Gal(H^+/F)\to \Gal(H/F))}=\frac{r}{2}. \qedhere    
  \end{split}
\]
  \end{proof}

% \begin{equation}
%   \label{eq:15}
% \{\sigma_i\mid 1\leq i \leq r\}=\{\sigma\in
% \Gal(H^+/F)\mid \sigma|_H=(\gra, H/F)\}. 
%   \end{equation}

% Via the isomorphism 

% the set $\{\sigma_i\mid 1\leq i \leq r\}$ is identified with the 

\section*{Acknowledgments}
Xue is partially supported by the National Natural Science Foundation of China grant No.~12271410.

%%%%%%%%%%%%%%%%%%%%%% AuCTeX  Tips %%%%%%%%%%%%
% (C-c %) to remove comments
% (C-c t p) to swap between latex and pdflatex
%%%%%%%%%%%%%%%%%%%%%%%%%%%%%%%%%%%%%%%%%%%%%%%%
\bibliographystyle{hplain}
\bibliography{TeXBiB}
\end{document}